\documentclass{amsart}
\usepackage{amsmath}
\usepackage{amsthm}
\usepackage{amsfonts}
\newtheorem{theorem}{Theorem}
\newtheorem{lemma}[theorem]{Lemma}
\author{Paul Myer Kominers and Scott Duke Kominers}
\title{Candy-Passing Games on General Graphs, I}
\thanks{The second author gratefully acknowledges the support of a Harvard Mathematics Department Highbridge Fellowship.}
\address{\newline\indent Student, Department of Mathematics, Massachusetts Institute of Technology\newline\indent c/o 8520 Burning Tree Road\newline \indent Bethesda, MD 20817}
\email{pkoms@mit.edu}
\address{\newline\indent Student, Department of Mathematics, Harvard University}
\email{kominers@fas.harvard.edu}

\subjclass[2000]{05C35 (Primary); 37B15 (Secondary)}
\keywords{candy-passing, chip-firing, stabilization, graph game}

\begin{document}
\maketitle
We let $G$ be an undirected graph and denote the vertex and edge sets of $G$ by $V(G)$ and
$E(G)$, respectively.  The
\emph{candy-passing game on $G$} is defined by the following rules:
\begin{itemize}
\item At the beginning of the game, $c>0$ candies are distributed
  among $|V(G)|$ students, each of whom is seated at some distinct vertex
  $v\in V(G)$.
\item A whistle is sounded at a regular interval.
\item Each time the whistle is sounded, every student who is able to do
  so passes one candy to each of his neighbors.  (If at the beginning of
  this step a student holds fewer candies than he has neighbors, he
  does nothing.)
\end{itemize}

The candy-passing game was first introduced by Tanton \cite{Candy
  Passing}, who defined the game for cyclic $G$.  Tanton and Wagon proved that if $G$ is an $n$-cycle then
  any candy-passing game on $G$ with fewer than $n$ candies terminates (see
  \cite{Candy Passing 2}).  The first author \cite{PKoms} also studied the end
  behavior of candy-passing games on such $G$, showing that if the number
  of candies $c$ is at least $3n-2$, then the configuration of candies
  eventually stabilizes.

Here, we undertake the first study of the candy-passing game on
arbitrary connected graphs $G$.  We obtain a general stabilization
result which encompasses the first author's \cite{PKoms} results for $c\geq 3n$.

\subsection*{Preliminaries}We call the interval between soundings of the whistle a \emph{round} of
candy-passing.   Dropping the student metaphor, we will treat the candy
piles as belonging to the vertices of the graph $G$.
If, after some round, the amount of candy held by a
given vertex will remain constant throughout all future rounds of the candy-passing
game, that vertex is said to have \emph{stabilized}.

We denote the degree of vertex $v\in V(G)$ by $\deg(v)$.  Clearly, if some vertex $v\in V(G)$ has $k\geq \deg(v)$ candies at the
beginning of a round, that vertex cannot end the round with more than
$k$ candies. 
Indeed, such a vertex will pass $\deg(v)$ pieces of candy to its
neighbors and can, at most, receive one piece of candy from each of its
$\deg(v)$ neighbors. 

Finally, we say that a
vertex $v\in V(G)$ is \emph{abundant} if it holds at least $2\deg(v)$ pieces of candy.  This definition implies:
\begin{lemma}
\label{Fixed Depletions}
After a finite number of rounds of the candy-passing game on $G$, the
set of abundant vertices of $G$ is fixed and each abundant vertex has
stabilized.
\end{lemma}
\begin{proof}
The total amount of candy on abundant vertices is nonincreasing.
Furthermore, whenever an abundant vertex loses candy, that total
decreases.  Since the total amount of candy on abundant vertices cannot
fall below zero, the amount of candy that can be lost by abundant
vertices must be finite, so that the set of abundant vertices and the
amount of candy on each such vertex must eventually become fixed.
\end{proof}

\subsection*{Main Result.}We may now prove our stabilization
  theorem:

\begin{theorem}\label{thmP}
Let $G$ be connected. In any candy-passing game on $G$ with
$$c\geq4|E(G)|-|V(G)|$$ candies, every vertex $v\in V(G)$ will eventually stabilize.
\end{theorem}
\begin{proof}
As a consequence of Lemma \ref{Fixed Depletions}, we may assume that all
candy that will be lost by abundant vertices over the course of the game
has been lost, as this must happen within finitely many rounds.  If
there are no abundant vertices at this point, then the 
condition $c\geq 4|E(G)|-|V(G)|$ implies $c=4|E(G)|-|V(G)|$ and that
every vertex $v\in V(G)$ has $2\deg(v)-1$ candies.  In this case, all the
vertices of $G$ have stabilized.

We now assume that at least one abundant vertex remains.  (Unless we are in the situation addressed in the prior paragraph, this is guaranteed by the condition $c\geq 4|E(G)|-|V(G)|=\sum_{v\in V(G)}\left(2\deg(v)-1\right)$.) This vertex has stabilized, and so it must be receiving candy from all of
its neighbors every round.  Each of its neighbors $v$, therefore, must hold at least
$\deg(v)$ pieces of candy every round.  It follows that these neighbors must eventually
stabilize, since these vertices pass candy every round and no such vertex may end a round with more candy than it
began with.   By an identical argument, the neighbors of these vertices must also pass candy every
round, and so they, too, must eventually stabilize.  As $G$ is connected, continuing
this argument shows that all the vertices of $G$ must eventually stabilize.
\end{proof}

\subsection*{Remarks}When $G$ is an $n$-cycle, the
condition $c\geq 4|E(G)|-|V(G)|$ is equivalent to the condition $c\geq
3n$.  Our Theorem \ref{thmP} generalizes the results of
\cite{PKoms} for $n$-cycles with at least $3n$ candies.  More generally,
if $G$ is connected and $k$-\emph{regular} then the condition of Theorem
\ref{thmP} simplifies to $c\geq (2k-1)|V(G)|$.

\end{document}